\documentclass[a4paper,11pt]{article}

\usepackage{amsmath,amsthm,amssymb,mathrsfs,latexsym,amsfonts}
\usepackage{graphicx,psfrag,epsfig}
\usepackage{bbm}
\usepackage[english]{babel}
\usepackage[latin1]{inputenc}
\usepackage{a4wide}

\newtheorem{theorem}{Theorem}[section]
\newtheorem{lemma}[theorem]{Lemma}
\newtheorem{proposition}[theorem]{Proposition}

\newcounter{paraga}[subsection]

\begin{document}

\def\MP{\,{<\hspace{-.5em}\cdot}\,}
\def\SP{\,{>\hspace{-.3em}\cdot}\,}
\def\PM{\,{\cdot\hspace{-.3em}<}\,}
\def\PS{\,{\cdot\hspace{-.3em}>}\,}
\def\EP{\,{=\hspace{-.2em}\cdot}\,}
\def\PP{\,{+\hspace{-.1em}\cdot}\,}
\def\PE{\,{\cdot\hspace{-.2em}=}\,}
\def\N{\mathbb N}
\def\C{\mathbb C}
\def\Q{\mathbb Q}
\def\R{\mathbb R}
\def\T{\mathbb T}
\def\A{\mathbb A}
\def\Z{\mathbb Z}
\def\demi{\frac{1}{2}}

\begin{titlepage}
\author{Abed Bounemoura~\footnote{CNRS-Université Paris Dauphine \& Observatoire de Paris (abedbou@gmail.com)} {} and Vadim Kaloshin~\footnote{University of Maryland at College Park (vadim.kaloshin@gmail.com)}}
\title{\LARGE{\textbf{Generic fast diffusion for a class of non-convex Hamiltonians with two degrees of freedom}}}
\end{titlepage}

\maketitle
\begin{center}
{\it To Yulij Ilyashenko on his 70th birthday with deep respect and admiration}
\end{center}

\begin{abstract}
In this paper, we study small perturbations of a class of non-convex integrable Hamiltonians with two degrees of freedom, and we prove a result of diffusion for an open and dense set of perturbations, with an optimal time of diffusion which grows linearly with respect to the inverse of the size of the perturbation. 
\end{abstract}
 
\section{Introduction and statement of the result}\label{s1}

\subsection{Introduction}\label{s11}

In this paper, we consider small perturbations of integrable Hamiltonian systems which are defined by a Hamiltonian function of the form
\[ H(\theta,I)=h(I)+\varepsilon f(\theta,I), \quad (\theta,I) \in \T^n \times \R^n, \quad 0 \leq \varepsilon <1, \]
where $n\geq 2$ is an integer and $\T^n=\R^n / \Z^n$. When $\varepsilon=0$, $H=h$ is integrable in the sense that the action variables $I(t)$ of all solutions $(\theta(t),I(t))$ of the system associated to $h$ are first integrals, $I(t)=I(0)$ for all times $t\in\R$. The sets $I=I_0$, for $I_0\in\R^n$, are thus invariant tori of dimension $n$ in the phase space $\T^n \times \R^n$, which moreover carry quasi-periodic motions with frequency $\omega(I_0)=\nabla h(I_0)$, that is $\theta(t)=\theta(0)+t\omega(I_0)$ modulo $\Z^n$. From now on we will assume that the small parameter $\varepsilon$ is non-zero, in which case the system defined by $H$ can be considered as an $\varepsilon$-perturbation of the integrable system defined by $h$.

In the sixties, Arnold conjectured that for a generic $h$, the following phenomenon should occur: ``for any points $I'$ and $I''$ on the connected level hypersurface of $h$ in the action space there exist orbits connecting an arbitrary small neighbourhood of the torus $I=I'$ with an arbitrary small neighbourhood of the torus $I=I''$, provided that $\varepsilon$ is sufficiently small and that $f$ is generic" (see \cite{Arn94}). This is a strong form of instability. A weaker form of this conjecture would be to ask for the existence of orbits for which the variation of the actions is of order one, that is bounded from below independently of $\varepsilon$ for all $\varepsilon$ sufficiently small. To support his conjecture, Arnold gave an example in \cite{Arn64} where this weaker form of instability is satisfied, with $n=2$, $h$ convex and $f$ a specific time-periodic perturbation (so this is equivalent to $n=3$, $h$ quasi-convex and $f$ a specific time-independent perturbation). The phenomenon highlighted in \cite{Arn64} is now known as Arnold diffusion.

\subsubsection{KAM stability}

Obstructions to Arnold diffusion, and to any form of instability in general, 
are widely known following the works of Kolmogorov and Arnold on the one hand, 
and the work of Nekhoroshev on the other hand. In \cite{Kol54}, Kolmogorov proved 
that for a non-degenerate $h$ and for all $f$, the system defined by $H$ still 
has many invariant tori, provided it is analytic and $\varepsilon$ is small enough. 
What he showed is that among the set of unperturbed invariant tori, there is 
a subset of positive measure (the complement of which has a measure going to zero
when $\varepsilon$ goes to zero) who survives any sufficiently small perturbation, 
the tori being only slightly deformed. The non-degeneracy assumption on $h$ is 
that at all points, the determinant of its Hessian matrix $\nabla^2 h(I)$ is non-zero. 
Then, under a different non-degeneracy assumption on $h$, namely that the determinant 
of the square matrix 
\[\begin{pmatrix} 
\nabla^2 h(I) & ^{t}\nabla h(I) \\
\nabla h(I) & 0 
\end{pmatrix}\]
is non-zero at all points, Arnold proved in \cite{Arn63a}, \cite{Arn63b} a similar 
statement but with a set of tori inside a fixed level hypersurface. In particular, 
for $n=2$, a level hypersurface is $3$-dimensional and the complement of the set of 
invariant $2$-dimensional tori is disconnected, and each connected component is 
bounded with a diameter going to zero as $\varepsilon$ goes to zero. As a consequence, 
it can be proved more precisely that for $n=2$ and if $h$ is non-degenerate in 
the sense of Arnold, along all solutions we have
\[ |I(t)-I(0)|\leq c\sqrt{\varepsilon}, \quad t\in\R,  \]
for some positive constant $c$. Therefore we have stability for all solutions and 
for all time. Now for any $n\geq 2$, and if $h$ is either Kolmogorov or Arnold 
non-degenerate, we have perpetual stability only for most solutions, those lying 
on invariant tori, and Arnold's example shows that this cannot be true for all 
solutions. The consequence of these results is that Arnold diffusion cannot exist 
for $n=2$ if $h$ is Arnold non-degenerate, and for $n\geq 2$ and $h$ Kolmogorov 
or Arnold non-degenerate, the unstable solution, if it exists, must live in a set 
of relatively small measure. 

\subsubsection{Nekhoroshev stability}

In an other direction, in the seventies Nekhoroshev proved (\cite{Nek77}, \cite{Nek79}) 
that for any $n\geq 2$, for a non-degenerate $h$ and for all $f$, along all solutions 
we have
\[ |I(t)-I(0)|\leq c_1\varepsilon^b, \quad |t|\leq \exp\left(c_2\varepsilon^{-a}\right),  \]  
for some positive constant $c_1,c_2,a$ and $b$, provided $\varepsilon$ is small enough 
and the system is analytic. So solutions which do not lie on invariant tori are stable 
not for all time, but during an interval of time which is exponentially long with 
respect to some power of the inverse of $\varepsilon$. The consequence on Arnold 
diffusion is that the time of diffusion, that is the time it takes for the action 
variables to drift independently of $\varepsilon$, is exponentially large. 
The integrable systems non-degenerate in the sense of Nekhoroshev, which are called 
{\it steep}, were originally quite complicated to define, but an equivalent definition 
was found in \cite{Ily86} and \cite{Nie06}: $h$ is steep if and only if its restriction 
to any affine subspace has only isolated critical points. Such functions can be proved 
to be generic in a rather strong sense (\cite{Nek73}), and the simplest 
(and also steepest) functions are the convex or quasi-convex ones (convex or 
quasi-convex functions are those for which the stability exponent $a$ in Nekhoroshev 
estimates is the best). Note that convex (respectively quasi-convex) functions are 
Kolmogorov (respectively Arnold) non-degenerate.

So the results of Kolmorogov, Arnold and Nekhoroshev restrict the possibility of 
diffusion, both in space and in time, at least provided the corresponding non-degeneracy 
assumptions are met. 

\subsubsection{Arnold's example and ``a priori" unstable systems}

Following the original insight of Arnold in \cite{Arn64}, much study have been devoted 
to perturbations of a special class of Hamiltonian systems, which are called 
``a priori" unstable, where these restrictions are much less stringent. We won't try 
to give a precise definition of ``a priori" unstable systems, but these systems are 
integrable in the larger sense of symplectic geometry (they have $n$ first integrals 
in involution and independent almost everywhere) but display hyperbolic features 
(typically they have a normally hyperbolic invariant manifold), and by opposition, 
the systems we are considering are called ``a priori" stable. These simpler 
``a priori" unstable systems are now well-understood, and many results confirm 
that instability occurs for a generic perturbation, see for instance 
\cite{Tre04}, \cite{CY04}, \cite{DDS06}, \cite{GR07}, \cite{Ber08}, \cite{DH09}, 
\cite{CY09}, and \cite{GR09}.

\subsubsection{``A priori" stable systems}

The situation for ``a priori" stable systems is much more complicated. 
In \cite{Mat04} (see also \cite{Mat12} for a recent corrected version), Mather 
announced a proof of Arnold conjecture in a special case, that is a strong form 
of Arnold diffusion for a generic time-dependent perturbation of a convex integrable 
Hamiltonian with $n=2$ (and also for a generic time-independent perturbation 
of a quasi-convex integrable Hamiltonian with $n=3$) based on his variational 
techniques. Mather never gave a complete proof of the announced results, 
but his work and unpublished preprints played a fundamental role in the subsequent 
developments. First, Bernard, Kaloshin and Zhang in \cite{BKZ} proved a weaker 
form of Arnold conjecture, still with the convexity requirement but for an arbitrary 
number of degrees of freedom. Then, Kaloshin and Zhang in \cite{KZ12} proved 
the strong form of Arnold conjecture, for $n=2$, $h$ convex and $f$ time-periodic. 
A different approach to this problem was proposed by Cheng (\cite{Che13}) and 
announced by Marco (\cite{Mar12a}, \cite{Mar12b}). 

\subsubsection{Role of normally hyperbolic invariant cylinders}

The central and common point in all these works, which was not present in the work 
of Mather, is the use of normally hyperbolic invariant manifolds as a ``skeleton" 
for the unstable orbits. Construction of such cylinders under generic hypothesis relies on resonant normal 
forms and the theory of normally hyperbolic manifolds, and they were already discussed in \cite{KZZ10} and \cite{Ber10}. 

On the other hand, most of these works do rely strongly on Mather's variational 
techniques, once the normally hyperbolic invariant manifolds have been constructed. 
It has to be noted that these variational techniques, and to a lesser extent, 
the existence of normally hyperbolic cylinders, use in an essential way 
the convexity assumption, so that none of these works apply to non-convex 
integrable Hamiltonians. 

\subsubsection{Non-convex systems}

It can be said that a typical non-degenerate integrable system (in the sense of 
Kolmogorov, Arnold or Nekhoroshev) is non-convex nor quasi-convex, but for 
these systems, essentially nothing is known: for the simplest integrable 
Hamiltonians $h$ which are non-convex nor quasi-convex but steep and 
non-degenerate (in the sense of Kolmogorov for $n\geq 2$ or Arnold for $n\geq 3$), 
it is not even known how to construct a single $f$ such that $H=h+\varepsilon f$ 
has unstable orbits. This is a bit paradoxical from the point of view of 
Nekhoroshev estimates, as the time of diffusion for perturbations of steep 
non-convex integrable Hamiltonians should be smaller and hence diffusion 
should be easier to observe.

Another evidence of the difficulty connected to the lack of convexity is the problem of existence of periodic orbits (not to mention, of course, the problem of extending Aubry-Mather theory). Recall that Bernstein and Katok proved, in \cite{BK87}, that given a periodic tori (filled with periodic orbits of common period) of a convex integrable Hamiltonian system with $n$ degrees of freedom, at least $n$ periodic orbits persist after a small perturbation, no matter how large is the period (that is, the threshold of the perturbation is independent of the period). Moreover, these periodic orbits are uniformly (with respect to the period) close to the unperturbed ones. For non-convex non-degenerate integrable Hamiltonian system, this was partly generalized in \cite{Che92}: one can still get the existence of at least $n$ periodic orbits but without the uniform estimates (one can find in \cite{Che92} an example due to Herman which shows that these uniform estimates are indeed not possible without convexity).

\subsubsection{Examples of non-convex non-steep Hamiltonians}\label{example}

Yet for some non-convex non-steep integrable Hamiltonians, the construction 
of examples of instability is much easier and has been known for a long time. A prototype of 
such an integrable Hamiltonian with two degrees of freedom, which can be found 
in \cite{Nek77} (but a completely analogous example, in a slightly different 
setting, was already considered in \cite{Mos60}), is given by 
$h(I_1,I_2)=\frac{1}{2}(I_1^2-I_2^2)$: letting 
$f(\theta_1,\theta_2)=(2\pi)^{-1}\sin(2\pi(\theta_1-\theta_2))$, the system 
$H=h+\varepsilon f$ admits the unstable solution $I(t)=(-\varepsilon t,\varepsilon t)$, 
$\theta(t)=-\frac{1}{2}(\varepsilon t^2,\varepsilon t^2)$. This Hamiltonian 
$h$ is obviously non-convex, but it is also non-steep since the restriction 
of $h$ to the lines $\{I_1 \pm I_2 =0\}$ is constant so this restriction has 
only critical points, which are thus non-isolated. Also it is degenerate in 
the sense of Arnold, so diffusion can and do already occur for $n=2$, even 
though it is non degenerate in the sense of Kolmogorov so that it admits 
many invariant tori (circles). Moreover, the time of diffusion in this example 
is the smallest possible, as it is linear with respect to the inverse of $\varepsilon$. Let us point out that for integrable systems with two degrees of freedom, the Arnold non-degeneracy condition is in fact equivalent to quasi-convexity which is also equivalent to steepness.  

It is obvious that the above example can be generalized to the case where $h$ is a quadratic form that has an isotropic vector with rational components (such a quadratic form is thus indefinite) for any number of degrees of freedom $n$, but in fact more is true. On the one hand, it was proved by Nekhoroshev in \cite{Nek79} that one can consider an even more general class of integrable Hamiltonians. Indeed, suppose there exist an affine subspace $L$ of $\R^n$, whose associated vector space is spanned by vectors with rational components, and a curve $\sigma : [0,1] \rightarrow L$ such that the gradient of the restriction of $h$ to $L$ vanishes identically along $\sigma$. Then one can construct an arbitrarily small perturbation $\varepsilon f$ such that the system $H=h+\varepsilon f$ has an orbit $(\theta(t),I(t))$ for which $I(0)=\sigma(0)$ and $I(\tau)=\sigma(1)$, where $\tau$ proportional to $\varepsilon^{-1}$. On the other hand, in the case where $h$ is a quadratic form with a rational isotropic vector, Herman had constructed examples of perturbation for which one can find a dense $G_\delta$ set of initial conditions leading to orbits whose action components are unbounded (note that such a quadratic form can be non-degenerate, so at the same time most of the orbits lie on invariant tori). We refer to \cite{Her92} for such examples and many other interesting examples related to non-convex integrable Hamiltonians. 

\subsubsection{Preliminary description of the class of Hamiltonians studied in the paper}

The examples above are rather specific. For instance, for the prototype of non-convex integrable Hamiltonian $h(I_1,I_2)=\frac{1}{2}(I_1^2-I_2^2)$, the perturbation $f$ does not depend on the action variables and more importantly, it depends only on a specific combination of the angular variables. 
The purpose of this paper is to investigate the question whether such a phenomenon 
remains true for a generic perturbation. We will show in Theorem~\ref{thm} that we have diffusion for a class of non-convex non-steep Hamiltonians $h$ with two 
degrees of freedom, which includes the example $h(I_1,I_2)=\frac{1}{2}(I_1^2-I_2^2)$ 
as a particular case, and for an open and dense set of perturbations, with a time of 
diffusion which is linear with respect to the inverse of $\varepsilon$. Under stronger assumptions, we will also prove a stronger statement of diffusion in Theorem~\ref{thm-potential}. The conditions 
defining this class of integrable Hamiltonians $h$ is the existence of a segment with 
rational slope contained in an energy level of $h$ such that the gradient $\nabla h$ does not vanish along this segment (these are the assumptions $(A.1)$ and $(A.2)
$ in \S\ref{s12}). For integrable Hamiltonians which are compatible with ``fast" 
diffusion (that is, with a time of diffusion which is linear with respect to 
the inverse of $\varepsilon$) for some perturbation, we expect these conditions 
to be quite sharp. The set of admissible perturbations is very easily described: 
we only required that some ``averaged" perturbation is a non-constant function. 
Moreover, we only require $h$ and $f$ to be of finite regularity. 

\subsubsection{Heuristic description of the proof}

Let us now explain the idea of the proof on the specific example $h(I_1,I_2)=I_1^2-I_2^2$. Up to a linear symplectic change of coordinates, we can equivalently consider $h(I_1,I_2)=I_1I_2$. If the perturbation $f$ depends only on $\theta_1$ and is non-constant, then we are essentially back to the example described in \S\ref{example}. Indeed, the equations of motion in this case are
\[ \dot{I}_1(t)=-\varepsilon f'(\theta_1(t)), \quad \dot{I}_2(t)=0, \quad \dot{\theta}_1(t)=I_2(t), \quad \dot{\theta}_2(t)=I_1(t) \]
so that if we choose an initial condition $(\theta(0),I(0))$ with $f'(\theta_1(0)) \neq 0$ (which is possible since $f$ is non constant) and $I_2(0)=0$, then $I_2(t)=0$ and $\theta_1(t)=\theta_1(0)$ for all time $t\in \R$ and so the $I_1$ variable drift with a speed of order $\varepsilon$:
\[ I_1(t)=I_1(0)-t\varepsilon f'(\theta_1(0)). \]
One should observe that this drift happens because the solution is locked in the resonance $\{I_2=0\}$ for all time.

Now for a general perturbation $f$ depending on $(\theta,I)$, on a suitable domain one can average the $\theta_2$ variables: more precisely the Hamiltonian can be conjugated to a Hamiltonian of the form $h(I)+\varepsilon \bar{f}(\theta_1,I)+\varepsilon^2 f'(\theta,I)$, where $\bar{f}$ is obtained from $f$ by averaging over $\theta_2$. This is close to the example we described before, except that $\bar{f}$ depends also on the action variables and there is a remainder of order $\varepsilon^2$, so we cannot find a solution which stays locked in the resonance $\{I_2=0\}$ for all time. However, it is still possible to find a solution which stays $\varepsilon$-close to the resonance $\{I_2=0\}$ for a time of order $\varepsilon^{-1}$ and this is sufficient to prove that the $I_1$ variable can drift independently of $\varepsilon$, provided of course that the function $\theta_1 \mapsto \bar{f}(\theta_1,I)$ is non-constant for a fixed value of $I$ (this is satisfied for an open and dense set of perturbations).

The main difficulty is that usually, the domain on which one can conjugated the Hamiltonian to the special form we described above (which is called a resonant normal form with a remainder) is $\varepsilon$-dependent and in the space of action, its diameter goes to zero as $\varepsilon$ goes to zero. We will actually prove (in Proposition~\ref{normal}) that in our situation,  we can construct such a conjugacy on a domain in the action space which contains a segment, in the $I_1$ direction, whose length is independent of $\varepsilon$, so that the existence of a drifting orbit for the normal form will actually yield the existence of a drifting orbit for the original Hamiltonian (this is Theorem~\ref{thm}). Assuming moreover that $f$ is action-independent and constructing a more accurate normal form with a remainder of order $\varepsilon^3$ instead of $\varepsilon^2$ (as in Proposition~\ref{normal2}), we will also show that the size of the drift can be as large as we want (this is Theorem~\ref{thm-potential}).

Let us point out that the main ingredient of the proof, which is the normal form, is similar in spirit to the normal form constructed in \cite{BKZ}. However, in \cite{BKZ} the normal form is used for another purpose (namely to construct a normally hyperbolic invariant cylinder which is then used to locate an unstable orbit) so they only need the remainder to be of order $\delta\varepsilon$, for some $\delta>0$ independent of $\varepsilon$. In our case, we need a stronger statement (a remainder of order at least $\varepsilon^2$) in order to derive the existence of an unstable orbit directly from the normal form.

\subsubsection{Prospects}

To conclude, let us note that the statement of Theorem~\ref{thm} gives 
a diffusion in a weak sense, that is the action variables drift independently 
of $\varepsilon$ for all $\varepsilon$ sufficiently small, but we cannot find 
an orbit which connects arbitrary neighbourhoods in the space of action. 
Also, for the moment, it is restricted to two degrees of freedom, which is 
the minimal number of degrees of freedom for which instability can occur 
for Arnold degenerate integrable systems. The normal form we used is in fact 
valid for any number of degrees of freedom, but in general it appears too weak 
to derive the result directly from it, and therefore we expect that additional 
restrictions on the set of admissible perturbations has to be imposed for 
more degrees of freedom. We plan to come back to these issues in a subsequent work.

\subsection{Main results}\label{s12}

\subsubsection{Geometric assumptions}

Given $R>0$, let $B_R$ 
be the closed ball of $\R^2$ of radius $R$ with respect to the supremum norm 
$|\,.\,|$, that is $B_R=\{(I_1,I_2)\in\R^2 \; | \; |I_1|\leq R, \; |I_2|\leq R\}$. 
Our integrable Hamiltonian $h$ will be a function $h : B_R \rightarrow \R$ of 
class $C^4$, which satisfy the following two conditions:

\bigskip

$(A.1)$ There exist a vector $k=(k_1,k_2)\in\Z^2\setminus\{0\}$, a constant $a\in\R$ and a closed segment $S \subseteq L \cap B_R$, where $L=\{(I_1,I_2) \in \R^2 \; | \; k_1I_1+k_2I_2+a=0\}$, such that the restriction of $h$ to $S$ is constant.

\bigskip

$(A.2)$ There exists a closed segment $S^* \subseteq S$, such that for all $I \in S^*$, $\nabla h(I) \neq 0$.

\bigskip

Such a segment $S$ is sometimes called a channel of superconductivity. Note that the condition $(A.1)$ obviously rules out convex functions, but it also rules out steep functions. Indeed, $(A.1)$ is equivalent to the assertion 
that the gradient of $h_{| L}$, the restriction of $h$ to $L$, vanishes identically on $S$, therefore the function $h_{| L }$ has a set of critical points which contains $S$ and hence is non-isolated. As for the condition $(A.2)$, it is a non-degeneracy assumption, as we want to avoid that the gradient of $h$ vanishes identically on $S$: note that $(A.2)$ is satisfied if there exists $I^*\in S$ such that $\nabla h(I^*) \neq 0$. The condition $(A.1)$ is crucial, whereas $(A.2)$ is somehow just technical, as we believe it can be removed in general. Following the terminology of~\cite{Bou12}, functions which do satisfy $(A.1)$ are functions which are not rationally steep. 

\subsubsection{Regularity assumptions}

Given a small parameter $0<\varepsilon<1$, our perturbation $\varepsilon f$ will be 
a ``generic" function $\varepsilon f : \T^2 \times B_R \rightarrow \R$ which is 
``small" for the $C^r$ topology, for $r$ sufficiently large. For an integer $r\geq 2$, let $C^r(\T^2 \times B_R)$ 
the space of $C^r$ function $f : \T^2 \times B_R \rightarrow \R$, which is Banach 
space with respect to the norm
\[ |f|_{C^r(\T^2 \times B_R)}=
\sup_{j\in \N^4,\; |j| \leq r}
\left(\sup_{(\theta,I)\in \T^n \times B_R}|\partial^{j}f(\theta,I)|\right) 
\]
where we have used the standard multi-index notation. We extend the definition of 
the $C^r$-norm for vector-valued functions 
$F=(f_1,\dots,f_m) : \T^2 \times B_R \rightarrow \R^m$, for an arbitrary integer $m\geq 1$, 
by setting 
\[ |F|_{C^r(\T^2 \times B_R,\R^m)}=
\sup_{1 \leq i \leq m}|f_i|_{C^r(\T^2 \times B_R)}. \]
Let us denote by $C^r_1(\T^2 \times B_R)$ the unit ball of $C^r(\T^2 \times B_R)$ 
with respect to this norm, that is
\[ 
C^r_1(\T^2 \times B_R)=\{f \in C^r(\T^2 \times B_R) \; | \; 
|f|_{C^r(\T^2 \times B_R)} \leq 1 \}. \]
Our perturbation $\varepsilon f$ will be such that $f$ belongs to 
an open and dense subset $\mathcal{F}_k^r$ of $C^r_1(\T^2 \times B_R)$, 
depending on the vector $k$ defined in $(A.1)$. For a given function $f \in C^r_1(\T^2 \times B_R)$ and a given $I^*$ in the interior of $S^*$, we define 
$\bar{f}_k^* \in C^r_1(\T^2)$ by
\[ \bar{f}_k^*(\theta)=\int_{0}^1f(\theta+tk,I^*)dt, \]
then $\mathcal{F}_k^r$ is defined by
\[ \mathcal{F}_k^r=\{f \in C^r_1(\T^2 \times B_R) \; | \; 
\exists \, I^* \in \mathrm{int}(S^*), \; \exists \, \theta^* \in \T^2, \; \partial_{\theta}\bar{f}_k^*(\theta^*)\neq 0\}.\]
In words, $\mathcal{F}_k^r$ is the subset of functions $f \in C^r_1(\T^2 \times B_R)$ 
such that, for some $I^*$ in the interior of $S^*$, the associated function $\bar{f}_k^*$ is non-constant: this is obviously an open and dense 
subset of $C^r_1(\T^2 \times B_R)$. Note that $\bar{f}_k^*$ is a function on $\T^2$, 
but by definition it is constant on the orbits of the linear flow of frequency $k$, 
hence it can be considered as being defined on the space of orbits (the leaf space) 
of this flow, which is diffeomorphic to $\T$.

\subsubsection{Statements}

We can finally state our first main result.

\begin{theorem}\label{thm}
Let $H=h+\varepsilon f$ be defined on $\T^2 \times B_R$, with $h \in C_1^4(B_R)$ 
satisfying $(A.1)$ and $(A.2)$ and $f \in \mathcal{F}_k^7$. Then there exist 
positive constants $C\geq 1$ and $c$, depending only on $h$, and positive constants $\varepsilon_0 \leq 1$ 
and $\delta \leq 1$ depending also on $f$, such that for any $0<\varepsilon \leq \varepsilon_0$, 
the Hamiltonian system defined by $H$ has a solution $(\theta(t),I(t))$ such that for $\tau=\delta \varepsilon^{-1}$, 
\[ |I(0)-I^*|\leq c\varepsilon, \quad |I(\tau)-I(0)| \geq C\delta^2. \]
Moreover, for all $t \in [0,\tau]$, we have $d(I(t),S^*) \leq c\varepsilon$ where $d$ is the distance induced by the supremum norm.
\end{theorem} 

It is a statement of diffusion for the action variables, in the sense that 
they have a variation along $S^* \subseteq S$ which is bounded from below independently of $\varepsilon$, 
for all $\varepsilon$ small enough. It has to be noted that the time of diffusion 
$\tau=\delta \varepsilon^{-1}$ is essentially optimal in the sense that for all 
$f \in C^2(\T^2 \times B_R) \cap C_1^1(\T^2 \times B_R)$, for all $\varepsilon>0$ 
and for all $0<\delta \leq 1$, we have
\[ 
|I(\tau)-I(0)|\leq \delta 
\]
for all solutions of $H=h+\varepsilon f$. In particular, for the solution given 
by Theorem~\ref{thm}, one has the inequalities
\[ C\delta^2 \leq |I(\tau)-I(0)|\leq \delta. \]

A stronger statement of diffusion can be reached, assuming that the perturbation is independent of the action variables and that it is slightly more regular. Indeed, let us define
\[ \mathcal{G}_k^r=\{f \in C^r_1(\T^2) \; | \; \exists \, \theta^* \in \T^2, \; \partial_{\theta}\bar{f}_k(\theta^*)\neq 0\}\]
where
\[ \bar{f}_k(\theta)=\int_{0}^1f(\theta+tk)dt. \]
Then we can state our second main result.

\begin{theorem}\label{thm-potential}
Let $H=h+\varepsilon f$ be defined on $\T^2 \times B_R$, with $h \in C_1^{10}(B_R)$ satisfying $(A.1)$ and $(A.2)$ and $f \in \mathcal{G}_k^{19}$. Then, given any two points $I'\in S^*$ and $I'' \in S^*$, there exists a positive constant $c$, depending only on $h$, and positive constants $\varepsilon_0 \leq 1$ and $\delta \leq 1$ depending also on $f$ and on the distance between $I'$ and $I''$, such that for any $0<\varepsilon \leq \varepsilon_0$, the Hamiltonian system 
defined by $H$ has a solution $(\theta(t),I(t))$ such that for $\tau\leq\delta\varepsilon^{-1}$, 
\[ 
|I(0)-I'| \leq c\varepsilon, \quad  
|I(\tau)-I''| \leq c\varepsilon. 
\]
Moreover, for all $t\in [0,\tau]$, we have $d(I(t),S^*)\le c\varepsilon$.
\end{theorem} 

The conclusion of Theorem~\ref{thm-potential} is indeed stronger than the conclusion of Theorem~\ref{thm}, as we not only have a variation along $S^* \subseteq S$ which is independent of $\varepsilon$, but we can also connect $\varepsilon$-neighborhoods of any two points in $S^*$. It should be noted that both theorems give a new obstruction to extend Nekhoroshev estimates in the non-steep case, even if one is willing to consider only a generic perturbation.

Concerning the dependence of the constants involved, the dependence on $h$ is only through $R$, the vector $k$ and the constant $a$ that appeared in $(A.1)$, the length of the segment $S^*$ that appeared in $(A.2)$ and a lower bound on the norm of $\nabla h(I)$ for $I \in S^*$ (which is positive by $(A.2)$), while the dependence on $f$ is through the absolute value of 
$\partial_{\theta}\bar{f}^*_k(\theta^*)$ (respectively $\partial_{\theta}\bar{f}_k(\theta^*)$) for Theorem~\ref{thm} (respectively for Theorem~\ref{thm-potential}) and the distance of $I^*$ to the boundary of $S^*$ for Theorem~\ref{thm}. We refer to Theorem~\ref{thm2} and Theorem~\ref{thm2-potential} in \S\ref{s21} for more concrete and precise statements.

\subsubsection{Comments on the regularity assumptions}

The first statement (Theorem~\ref{thm}) requires the integrable part to be $C^4$ and the perturbation to be $C^7$, while the second statement (Theorem~\ref{thm-potential}) requires the integrable part to be $C^{10}$ and the perturbation to be $C^{19}$. Theses regularity assumptions are far from being optimal, and no efforts were made to improve them.

For instance, using analytic approximations as in \cite{BKZ} or polynomial approximations, it is certainly possible to lower these regularities. Moreover, concerning the first statement, one can use a different method which would give the same result assuming only that the integrable part and the perturbation are of class $C^3$ (unfortunately, this method cannot be applied directly to prove the second statement).

\subsubsection{Comments on the geometric assumptions}

Let us now discuss some particular cases of functions $h$ satisfying $(A.1)$ and $(A.2)$, and therefore for which one has diffusion for a generic perturbation. As we will explain later, we can always assume without loss of generality that $a=0$ in $(A.1)$, and upon adding an irrelevant additive constant, we can assume that the restriction of $h$ to $S$ is identically zero. 

For a linear Hamiltonian $h(I)=\omega\cdot I$, it follows that $(A.1)$ and $(A.2)$ are satisfied if and only if $\omega$ is resonant, that is $l\cdot\omega=0$ for some $l\in\Z^2\setminus\{0\}$, and $\omega$ is non-zero. On the other hand, if $\omega$ is non-resonant, it follows from~\cite{Bou12} that the statement of Theorem~\ref{thm} cannot be true since for all sufficiently small perturbation, one has stability for an interval of time which is strictly larger than $[-\tau,\tau]$ with $\tau$ as above. In particular, if $\omega$ is Diophantine, one has stability for an interval of time which is exponentially long with respect to $\varepsilon^{-1}$, up to an exponent depending only on the Diophantine exponent of $\omega$. 

Now for a quadratic Hamiltonian $h(I)=AI\cdot I$ where $A$ is a $2$ by $2$ symmetric matrix, $(A.1)$ and $(A.2)$ are satisfied if and only if there exists a vector $l\in \Z^2\setminus\{0\}$ such that $Al\cdot l=0$ and $Al\neq 0$. Assuming that $A$ is diagonal, its eigenvalues have to be of different sign, and writing $h(I)=\alpha_1^2I_2-\alpha_2^2I_2^2$, $(A.1)$ and $(A.2)$ are satisfied if and only if $\alpha_1 \neq 0$, $\alpha_2 \neq 0$ and $\alpha_2/\alpha_1 \in \Q$. The example described in the introduction corresponds to $\alpha_1=\alpha_2=1$. On the other hand, one knows that if $\alpha_2/\alpha_1$ is irrational, the statement of Theorem~\ref{thm2} cannot be true for any sufficiently small perturbation for the same reason as above: for instance, if $\alpha_2/\alpha_1$ is a Diophantine number, the quadratic Hamiltonian falls into the class of Diophantine steep functions introduced in \cite{Nie07} and it follows from results in \cite{Nie07} or \cite{BN09} that such Hamiltonians are stable for an 
exponentially long interval of time. 

Note that in these two special cases, the condition $(A.2)$, which amounts to $\omega \neq 0$ in the first case and $Al\neq 0$ in the second case, can be easily removed.

We already explained that the time of diffusion $\tau$ is in some sense optimal, regardless of the integrable Hamiltonian $h$. Now we believe that if we fix the time of diffusion, the condition $(A.1)$ on the integrable Hamiltonian $h$ is also in some sense optimal, as if $h$ does not satisfy this assumption, one can have diffusion but with a time strictly greater than $\tau$. This is indeed the case for linear or quadratic integrable Hamiltonians as we described above, and the general case is conjectured in \cite{Bou12}.

\section{Proof of Theorem~\ref{thm} and Theorem~\ref{thm-potential}}\label{s2}

In \S\ref{s21}, we will perform some preliminary transformations to reduce Theorem~\ref{thm} and Theorem~\ref{thm-potential} to equivalent but more concrete statements, which are Theorem~\ref{thm2} and Theorem~\ref{thm2-potential}. Theorem~\ref{thm2} and Theorem~\ref{thm2-potential} will be proved in \S\ref{s23}, based on normal form results which are stated and proved in \S\ref{s22}.  

\subsection{Preliminary reductions}\label{s21}

\subsubsection{Preliminary transformations}

First we may assume that the line $L$ in $(A.1)$ passes through the origin, that is $L=\{(I_1,I_2) \in \R^2 \; | \; k_1I_1+k_2I_2=0\}$: indeed, we can always find a translation of the action variables $T : \R^2 \rightarrow \R^2$ such that $T$ sends $\{(I_1,I_2) \in \R^2 \; | \; k_1I_1+k_2I_2=0\}$ to $\{(I_1,I_2) \in \R^2 \; | \; k_1I_1+k_2I_2+a=0\}$, and since the map $\Phi_T(\theta,I)=(\theta,TI)$ is symplectic, the statement holds true for $H$ if and only if it holds true for $H \circ \Phi_T$, up to constants depending on $a$.

Then we can suppose that the components of the vector $k=(k_1,k_2) \in \Z^2\setminus\{0\}$ are relatively prime, since changing $k$ by $k/p$, where $p$ is the greatest common divisor of $k_1$ and $k_2$, does not change the definition of $L$. Hence we may assume that in fact $k=e_2=(0,1)$, that is $L=\{(I_1,I_2) \in \R^2 \; | \; I_2=0\}$: indeed we can always find a matrix $M \in GL_2(\Z)$ such that its second row is $k$, hence $Me_2=k$ and $^{t}M^{-1}$ sends $\{(I_1,I_2) \in \R^2 \; | \; I_2=0\}$ to $\{(I_1,I_2) \in \R^2 \; | \; k_1I_1+k_2I_2=0\}$. The map $\Phi_M(\theta,I)=(M\theta,^{t}M^{-1}I)$ is well-defined since $M\T^2=\T^2$, and it is symplectic, so the statement holds true for $H$ if and only if it holds true for $H \circ \Phi_M$, up to constants depending on $k$.    

Note that the symplectic transformations $\Phi_T$ and $\Phi_M$ do change the domain $B_R$ in the space of actions, but to simplify the notations, we will assume that the latter is fixed.

\subsubsection{Simplified assumptions}

Now for all $I=(I_1,I_2) \in B_R$, let us write 
\[ \nabla h(I)=\omega(I)=(\omega_1(I),\omega_2(I))=(\omega_1(I_1,I_2),\omega_2(I_1,I_2)) \in \R^2. \]
Since $L=\{(I_1,I_2) \in \R^2 \; | \; I_2=0\}$, $S$, which by definition is a closed segment contained in $L \cap B_R$, is of the form $S=\{(I_1,0)\in \R^2 \; | \; I_1 \in S_1 \}$ 
where $S_1$ is a closed segment of $\R$ contained in $[-R,R]$. Similarly, $S^*=\{(I_1,0)\in \R^2 \; | \; I_1 \in S_1^* \}$ where $S_1^*$ is contained in the interior of $S_1$. The condition $(A.1)$ is then obviously equivalent to $\partial_{I_1}h(I_1,0)=\omega_1(I_1,0)=0$ for all $I_1 \in S_1$, while the condition $(A.2)$ is that $\omega_2(I_1,0)\neq 0$ for all $I_1 \in S_1^*$. Changing $H$ to $-H$ if necessary and reversing the time accordingly, we may assume that $\omega_2(I_1,0) \geq \varpi >0$ for all $I_1 \in S_1^*$. 

We can eventually formulate simplified conditions, that we call $(B.1)$ and $(B.2)$: 

\bigskip

$(B.1)$ There exists a closed segment $S_1 \subseteq [-R,R]$ such that for all $I_1 \in S_1$, we have $\omega_1(I_1,0)=0$. 

\bigskip

$(B.2)$ There exist a closed segment $S_1^* \subseteq S_1$ and $\varpi>0$ such that for all $I_1 \in S_1^*$, we have $\omega_2(I_1,0)\geq \varpi$.

\bigskip

Then the definition of $\mathcal{F}_{e_2}^r$ also simplifies: one easily check 
that for $f \in C^r_1(\T^2 \times B_R)$, we have, for some $I^*=(I_1^*,0)$ in the interior of $S^*$, $\bar{f}_{e_2}^* \in C_1^r(\T)$ where 
\[ \bar{f}_{e_2}^*(\theta_1)=\int_{\T}f(\theta_1,\theta_2,I^*)d\theta_2 \]
so that $f \in \mathcal{F}^r_{e_2}$ if and only if there exist $I^*=(I_1^*,0)$ in the interior of $S^*$ and $\theta_1^* \in \T$ for which $\partial_{\theta_1} \bar{f}_{e_2}^*(\theta_1^*) \neq 0$. The definition of $\mathcal{G}_{e_2}^r$ is analogous. For simplicity, we write $\bar{f}_{e_2}^*=\bar{f}^*$ and $\mathcal{F}_{e_2}^r=\mathcal{F}^r$, $\bar{f}_{e_2}=\bar{f}$ and $\mathcal{G}_{e_2}^r=\mathcal{G}^r$, and for $f \in \mathcal{F}^r$ (respectively $f \in \mathcal{G}^r$), we denote by $\lambda$ a lower bound on the absolute value of $\partial_{\theta_1} \bar{f}^*(\theta_1^*)$ (respectively a lower bound on 
the absolute value of $\partial_{\theta_1} \bar{f}(\theta_1^*)$). For Theorem~\ref{thm}, we denote by $\delta^*$ the distance of $I_1^*$ to the boundary of $S_1^*$, and for Theorem~\ref{thm-potential}, we denote by $\rho$ the distance between $I_1'\in S_1^*$ and $I_1''\in S_1^*$, where $I'=(I_1',0)$ and $I''=(I_1'',0)$.

\subsubsection{Simplified statements}

From the previous discussion, it follows that Theorem~\ref{thm} and Theorem~\ref{thm-potential} are implied by the following statements.

\begin{theorem}\label{thm2}
Let $H=h+\varepsilon f$ be defined on $\T^2 \times B_R$, with $h \in C_1^4(B_R)$ 
satisfying $(B.1)$ and $(B.2)$ and $f \in \mathcal{F}^7$. Then there exist positive constants $C \geq 1$ and $c$ depending only on $R$, the length of $S_1^*$ and $\varpi$, and a positive constant $\varepsilon_0 \leq 1$ depending also on $\lambda$, such that for any $0<\varepsilon \leq \varepsilon_0$, if we set $\delta=\min\{\lambda(4C)^{-1},\delta^*\}$, 
the Hamiltonian system defined by $H$ has a solution $(\theta(t),I(t))$ 
such that for $\tau=\delta\varepsilon^{-1}$,
\[ |I_1(0)-I_1^*|\leq c\varepsilon, \quad |I_1(\tau)-I_1(0)| \geq C\delta^2. \]
Moreover, for all $t \in [0,\tau]$, we have $|I_2(t)|\leq c\varepsilon$ and $d(I_1(t),S_1^*)\leq c\varepsilon$.
\end{theorem}

\begin{theorem}\label{thm2-potential}
Let $H=h+\varepsilon f$ be defined on $\T^2 \times B_R$, with $h \in C_1^{10}(B_R)$ 
satisfying $(B.1)$ and $(B.2)$ and $f \in \mathcal{F}^{19}$. Then there exists a positive constant $c$, depending only on $R$, the length of $S_1^*$, and $\varpi$, and a positive constant $\varepsilon_0 \leq 1$ depending also on $\lambda$ and $\rho$, such that for any $0<\varepsilon \leq \varepsilon_0$, if we set $\delta=2\rho\lambda^{-1}$, 
the Hamiltonian system defined by $H$ has a solution $(\theta(t),I(t))$ 
such that for $\tau\leq \delta\varepsilon^{-1}$,
\[ |I_1(0)-I_1'|\leq c\varepsilon, \quad |I_1(\tau)-I_1''| \leq c\varepsilon. \]
Moreover, for all $t \in [0,\tau]$, we have $|I_2(t)|\leq c\varepsilon$ and $d(I_1(t),S_1^*)\leq c\varepsilon$.
\end{theorem}

\subsection{Normal forms}\label{s22}

\subsubsection{Domain of the normal forms}

The main ingredient of the proofs of Theorem~\ref{thm2} and Theorem~\ref{thm2-potential} will be
normal forms on a domain which, in the space of action, is centred around $S^*$. In particular,
in the direction given by the first action variables $I_1$, it contains the segment $S_1^*$ whose length is independent of $\varepsilon$. For a constant $\kappa>0$ to be determined later (in the proof of Proposition~\ref{normal}), let us consider 
the $\kappa\varepsilon$-neighbourhood of $S^*$ in $\R^2$ 
\[ 
S^*(\kappa\varepsilon)=\{I \in \R^2 \; | \; d(I,S^*)\leq 
\kappa\varepsilon\}=\{(I_1,I_2) \in \R^2 \; | \; d(I_1,S_1^*) \leq \kappa\varepsilon, 
\; |I_2|\leq \kappa\varepsilon \}. \]
We also define the domain $\mathcal{D}^*(\kappa\varepsilon)=\T^2 
\times S^*(\kappa\varepsilon)$. Let $r$ be an integer (we will choose $r=7$ for Theorem~\ref{thm2} and $r=19$ for Theorem~\ref{thm2-potential}, but for convenience we consider it as a free parameter for the moment).

In the sequel, to avoid 
cumbersome notations, when convenient we will use a dot $\cdot$ 
in replacement of any constant depending only on $r$, $R$, the length of $S_1^*$ and 
$\varpi$, that is for any two quantities $u$ and $v$, an expression 
$u \MP v$ means that there exists a constant $c$ depending only on $r$, $R$, the length of $S_1^*$ and $\varpi$ such that $u \leq c v$. Similarly, we will use the notation $u \EP v$. For instance, we will have $\kappa \EP 1.$

To simplify the exposition further, for any integer $j \leq r$ we will simply denote by $|\,.\,|_j$ the $C^j$ norm of a function or a vector-valued function, without referring to its domain of definition nor to the domain where it takes values.

\subsubsection{One-step normal form}

\begin{proposition}\label{normal}
Let $H=h+\varepsilon f$ be defined on $\T^2 \times B_R$, $l\geq 1$ an integer, and assume that $h \in C_{1}^{l+3}(B_R)$ 
satisfies $(B.1)$ and $(B.2)$ and 
$f \in C_1^r(\T^2 \times B_R)$, for $r\geq 2(l+1)+3$. 
Assume that $\varepsilon \MP 1$. Then there exists 
a symplectic embedding 
\[
\Phi : \mathcal{D}^*(\kappa \varepsilon/2) \rightarrow 
\mathcal{D}^*(\kappa \varepsilon) 
\]
of class $C^{l+1}$ such that
\[ 
H \circ \Phi = 
h+ \varepsilon \bar{f}+ \varepsilon^2 f', \quad 
\bar{f}(\theta_1,I)=\int_{\T}f(\theta_1,\theta_2,I)d\theta_2 
\]
and we have the following estimates
\[ |\Phi-\mathrm{Id}|_{0} \leq \kappa\varepsilon/2, \quad |f'|_{l} \MP 1. \] 
\end{proposition}

The proof of this proposition uses some elementary estimates which are recalled in the Appendix~\ref{app}.

\begin{proof}[Proof of Proposition~\ref{normal}]
First of all, since $\varepsilon \MP 1$, we can assume that $S^*(\kappa\varepsilon)$ is included in $B_R$. For a function $\chi : \mathcal{D}^*(\kappa\varepsilon) \rightarrow \R$ of class $C^{l+2}$ to be chosen below, the transformation $\Phi$ in the statement will be obtained as the time-one map of the Hamiltonian flow generated by $\varepsilon \chi$. Let $X_{\varepsilon\chi}$ be the Hamiltonian vector field generated by $\varepsilon\chi$, and $X_{\varepsilon\chi}^t$ the time-$t$ map. Assuming that $X_{\varepsilon\chi}^t$ is well-defined on $\mathcal{D}^*(\kappa\varepsilon/2)$ for $|t|\leq 1$, let $\Phi=X_{\varepsilon\chi}^1$. Using the relation
\[ \frac{d}{dt}\left(G \circ X_{\varepsilon\chi}^t\right)=\varepsilon\{G,\chi\}\circ X_{\varepsilon\chi}^t \]
for an arbitrary function $G$, and writing
\[ H \circ \Phi=h \circ \Phi + \varepsilon f \circ \Phi \]
we can apply Taylor's formula to the right-hand side of the above equality, at order two for the first term and at order one for the second term, to get
\begin{eqnarray}\label{dvpt} \nonumber
H \circ \Phi & = & h + \varepsilon \{h,\chi\}+\varepsilon^2 \int_{0}^{1}(1-t)\{\{h,\chi\},\chi\}\circ X_{\varepsilon\chi}^t dt + \varepsilon f + \varepsilon^2 \int_{0}^{1}\{f,\chi\}\circ X_{\varepsilon\chi}^t dt \\ \nonumber
& = & h+\varepsilon (\{h,\chi\}+f)+\varepsilon^2 \int_{0}^{1} \{(1-t)\{h,\chi\}+f,\chi\}\circ X_{\varepsilon\chi}^t dt \\ 
& = & h+\varepsilon \bar{f}+\varepsilon (\{h,\chi\}+f-\bar{f})+\varepsilon^2 \int_{0}^{1} \{(1-t)\{h,\chi\}+f,\chi\}\circ X_{\varepsilon\chi}^t dt
\end{eqnarray} 
where $\bar{f}$ is the function defined in the statement. It would be natural to choose $\chi$ to solve the equation $\{h,\chi\}+g=0$ where $g=f-\bar{f}$, which can be written again as $\{\chi,h\}=g$, but we will only solve this equation approximatively.

We expand $f$ in Fourier series with respect to the variables $\theta$:
\[ f(\theta,I)=\sum_{k\in\Z^2}f_k(I)e^{2\pi ik\cdot\theta}=\sum_{(k_1,k_2)\in\Z^2}f_{(k_1,k_2)}(I)e^{2\pi i(k_1\theta_1+k_2\theta_2)} \]
and for a parameter $K\geq 1$ to be chosen below, we write
\[ f(\theta,I)=f_K(\theta,I)+f^K(\theta,I)=\sum_{k\in\Z^2, \; |k|\leq K}f_k(I)e^{2\pi ik\cdot\theta}+\sum_{k\in\Z^2, \; |k|>K}f_k(I)e^{2\pi ik\cdot\theta}. \]
Instead of solving the equation $\{\chi,h\}=g$ with $g=f-\bar{f}$, we will actually choose $\chi$ to solve the equation $\{\chi,h\}=g_K$ where $g_K=f_K-\bar{f}_K$. Observe that
\[ \bar{f}(\theta_1,I)=\sum_{k_1\in \Z}f_{(k_1,0)}(I)e^{2\pi ik_1\theta_1} \]
and hence
\[ \bar{f}_K(\theta_1,I)=\sum_{k_1\in \Z, \; |k_1|\leq K}f_{(k_1,0)}(I)e^{2\pi ik_1\theta_1}. \]
It is then easy to check that the solution of $\{\chi,h\}=g_K$, which can be written again as
\begin{equation}\label{eqhomo}
\omega(I)\cdot\partial_\theta \chi(\theta,I)=g_K(\theta,I), \quad (\theta,I)\in \mathcal{D}^*(\kappa\varepsilon)
\end{equation}
is given by $\chi(\theta,I)=\sum_{k\in \Z^2, \; |k|\leq K}\chi_k(I)e^{2\pi i k\cdot\theta}$ where
\begin{equation}\label{sol}
\chi_k(I)=
\begin{cases}
(2\pi i k\cdot \omega(I))^{-1}f_k(I), & |k| \leq K, \; k_2 \neq 0 \\
0, & |k| \leq K, \; k_2=0. 
\end{cases}
\end{equation}
Since $h$ is $C^{l+3}$, $\chi$ is $C^{l+2}$ provided we can prove that $k\cdot \omega(I)$ is non-zero, for $|k|\leq K$ such that $k_2 \neq 0$ and $I \in S^*(\kappa\varepsilon)$. By definition, given $I \in S^*(\kappa\varepsilon)$ we can find $\tilde{I} \in S^*$ such that $|I-\tilde{I}|\leq \kappa\varepsilon$ and therefore $|\omega(I)-\omega(\tilde{I})|\leq \kappa\varepsilon$ since $h\in C_{1}^{l+3}(B_R)$. For any $k\in \Z^2$ such that $k_2 \neq 0$, by $(B.1)$ we have $\omega_1(\tilde{I})=\omega_1(\tilde{I}_1,0)=0$ (since $\tilde{I}=(\tilde{I}_1,0)\in S^* \subseteq S$). Moreover by $(B.2)$ we have $\omega_2(\tilde{I}_1,0) \geq \varpi$ (since $\tilde{I}_1 \in S_1^*$), hence 
\[|k\cdot \omega(\tilde{I})|=|k\cdot\omega(\tilde{I}_1,0)|=|k_1\omega_1(\tilde{I}_1,0)+k_2\omega_2(\tilde{I}_1,0)|=|k_2|\omega_2(\tilde{I}_1,0)\geq \varpi.\] 
It follows that for $k\in \Z^2$ such that $k_2 \neq 0$ and $|k|\leq K$, and for $I \in S^*(\kappa\varepsilon)$, we have
\begin{equation}\label{smalld}
|k\cdot \omega(I)|\geq  |k\cdot \omega(\tilde{I})|-|k||\omega(I)-\omega(\tilde{I})| \geq \varpi - K\kappa\varepsilon \geq \varpi/2
\end{equation}    
provided we define $K=\varpi(2\kappa\varepsilon)^{-1}\EP \varepsilon^{-1}$. Now $f \in C_1^r(\T^2 \times B_R)$, so an integration by parts gives that $|f_k|_j \MP |k|^{j-r}|f|_r \MP |k|^{j-r}$ for any $j\leq r$, and therefore using~\eqref{sol},~\eqref{smalld} and Leibniz formula (inequality~\eqref{eg2} 
of Appendix~\ref{app}), we obtain $|\chi_k|_{l+1} \MP |k|^{l+1-r}$. Since $r \geq 2(l+1)+3$, we can therefore bound the $C^{l+1}$ norm of $\chi$ independently of $K$ as
\[ |\chi|_{l+1} \MP \sum_{k\in \Z^2, \; |k|\leq K}|\chi_k|_{l+1}|k|^{l+1} \MP \sum_{k\in \Z^2, \; |k|\leq K}|k|^{l+1-r}|k|^{l+1} \MP \sum_{k\in \Z^2}|k|^{-3} \MP 1   \]
and so
\begin{equation}\label{gene}
|\varepsilon\chi|_{l+1} \leq \gamma\varepsilon, \quad \gamma\EP 1.
\end{equation}
It is easy to see that $\gamma$ is independent of $\kappa$, so we now choose $\kappa=2\gamma\EP 1$, and as $\varepsilon \MP 1$, we can apply Lemma~\ref{tech} of Appendix~\ref{app}: for all $|t|\leq 1$, 
$X^t_{\varepsilon\chi} : \mathcal{D}^*(\kappa\varepsilon/2) \rightarrow 
\mathcal{D}^*(\kappa\varepsilon)$ is a well-defined symplectic embedding of 
class $C^{l+1}$ with the estimates
\begin{equation}\label{esttech} 
|X_{\varepsilon\chi}^t-\mathrm{Id}|_{0} 
\leq \kappa\varepsilon/2, \quad 
|X_{\varepsilon\chi}^t|_{l} \MP 1. 
\end{equation}
In particular, the first estimate of~\eqref{esttech} gives
\[ |\Phi-\mathrm{Id}|_{0} \MP \varepsilon.\]

Now from the equalities~\eqref{dvpt} and~\eqref{eqhomo} we can write
\[ H \circ \Phi= h+\varepsilon \bar{f}+\varepsilon f^K - \varepsilon \bar{f}^K+\varepsilon^2 \int_{0}^{1} \{(t-1)g_K+f,\chi\}\circ X_{\varepsilon\chi}^t dt \]
so that, if we set
\[ f'=\varepsilon^{-1}f^K-\varepsilon^{-1}\bar{f}^K+\int_{0}^{1} \{(t-1)g_K+f,\chi\}\circ X_{\varepsilon\chi}^t dt=\varepsilon^{-1}f^K-\varepsilon^{-1}\bar{f}^K+R, \]
then
\[ H \circ \Phi = h+\varepsilon \bar{f} + \varepsilon^2 f'. \]
It remains to estimate $f'$. Using the fact that $|f_k|_j \MP |k|^{j-r}|f|_r \MP |k|^{j-r}$ for any $j\leq r$, one easily obtain
\begin{equation}\label{rem1}
\varepsilon^{-1}|f^K|_l \MP \varepsilon^{-1}K^{l+2-r} \MP 1
\end{equation}
by definition of $K$. Similarly
\begin{equation}\label{rem2}
\varepsilon^{-1}|\bar{f}^K|_l \MP 1.
\end{equation}
Then we have
\begin{eqnarray}\label{rem3} \nonumber
|R|_{l} & \MP &  |\{(t-1)g_K+f,\chi\}|_{l}|X_{\varepsilon\chi}^t|_{l}^{l} \\ \nonumber
& \MP & |\{(t-1)g_K+f,\chi\}|_{l} \\ \nonumber
& \MP &  |(t-1)g_K+f|_{l+1}|\chi|_{l+1} \\
& \MP & 1 
\end{eqnarray}
where we have used Faa di Bruno formula (inequality~\eqref{Faa} of Appendix~\ref{app}), the last part of~\eqref{esttech}, the inequality~\eqref{poisson} of Appendix~\ref{app} and the fact that $|g_K|_{l+1} \MP 1$, $|f|_{l+1} \MP 1$ and $|\chi|_{l+1} \MP 1$. The estimates~\eqref{rem1},~\eqref{rem2} and~\eqref{rem3} implies that
\[ |f'|_l \MP 1 \]
which concludes the proof. 
\end{proof}

\subsubsection{Two-steps normal form}

\begin{proposition}\label{normal2}
Let $H=h+\varepsilon f$ be defined on $\T^2 \times B_R$, $l\geq 1$ and $l'\geq 1$ integers, and assume that $h \in C_{1}^{l+3}(B_R)$ 
satisfies $(B.1)$ and $(B.2)$ and 
$f \in C_1^r(\T^2 \times B_R)$, for $r\geq 2(l+1)+3$ and $l \geq 2(l'+1)+3$. 
Assume that $\varepsilon \MP 1$. Then there exists 
a symplectic embedding 
\[
\Phi : \mathcal{D}^*(\kappa \varepsilon/4) \rightarrow 
\mathcal{D}^*(\kappa \varepsilon) 
\]
of class $C^{l'+1}$ such that
\[ 
H \circ \Phi = 
h+ \varepsilon \bar{f}+ \varepsilon^2 \bar{f}'+\varepsilon^3 f'', \quad 
\bar{f}(\theta_1,I)=\int_{\T}f(\theta_1,\theta_2,I)d\theta_2, \quad \bar{f}'(\theta_1,I)=\int_{\T}f'(\theta_1,\theta_2,I)d\theta_2 
\]
and we have the following estimates
\[ |\Phi-\mathrm{Id}|_{0} \leq 3\kappa\varepsilon/4, \quad |f'|_{l} \MP 1, \quad |f''|_{l'} \MP 1. \] 
\end{proposition}

The proof of Proposition~\ref{normal2} consists essentially of applying twice Proposition~\ref{normal}; in particular the estimates are analogous so they will not be repeated below.

\begin{proof}[Proof of Proposition~\ref{normal2}]
First of all, the assumptions allow us to apply Proposition~\ref{normal}: there exists 
a symplectic embedding 
\[
\Phi_1 : \mathcal{D}^*(\kappa \varepsilon/2) \rightarrow 
\mathcal{D}^*(\kappa \varepsilon) 
\]
of class $C^{l+1}$ such that
\[ 
H_1=H \circ \Phi_1 = 
h+ \varepsilon \bar{f}+ \varepsilon^2 f', \quad 
\bar{f}(\theta_1,I)=\int_{\T}f(\theta_1,\theta_2,I)d\theta_2 
\]
and we have the following estimates
\[ |\Phi_1-\mathrm{Id}|_{0} \leq \kappa \varepsilon/2, \quad |f'|_{l} \MP 1. \]  
Now consider the Hamiltonian $H_1=h+ \varepsilon \bar{f}+ \varepsilon^2 f'$ defined on $\mathcal{D}^*(\kappa \varepsilon/2)$ and of class $C^{l+1}$. The transformation $\Phi$ in the statement will be obtained as a composition $\Phi=\Phi_1 \circ \Phi_2$, where $\Phi_2$ will be the time-one map of the Hamiltonian flow generated by $\varepsilon^2 \chi$, for some function $\chi : \mathcal{D}^*(\kappa \varepsilon/2) \rightarrow \R$ of class $C^{l'+2}$ to be determined. As before, we write
\[ H_1 \circ \Phi_2 =h \circ \Phi_2 + (\varepsilon \bar{f}+\varepsilon^2 f') \circ \Phi_2 \]
and by a Taylor expansion we have
\begin{eqnarray*}
H_1 \circ \Phi_2 & = & h + \varepsilon^2 \{h,\chi\}+\varepsilon^4 \int_{0}^{1}(1-t)\{\{h,\chi\},\chi\}\circ X_{\varepsilon\chi}^t dt \\
& + & \varepsilon \bar{f} + \varepsilon^2 f'+ \varepsilon^3 \int_{0}^{1}\{\bar{f}+\varepsilon f',\chi\}\circ X_{\varepsilon\chi}^t dt \\ 
& = & h+\varepsilon \bar{f}+\varepsilon^2 \bar{f}'+\varepsilon^2 (\{h,\chi\}+f'-\bar{f}')+\varepsilon^3 R.
\end{eqnarray*} 
As before also, we will choose $\chi$ to solve the equation $\{\chi,h\}=g'_K$ where $g_K=f'_K-\bar{f}'_K$ and $K \EP \varepsilon^{-1}$. Note that $H_1$ is of class $C^{l+1}$, but we have a bound only on the $C^l$ norm of $f'$, and as $l \geq 2(l'+1)+3$, it can be proved that $\chi$ is $C^{l'+2}$ with
\[ |\varepsilon^2 \chi|_{l'+1} \MP \varepsilon^2 \]
and hence
\[ |\Phi_2-\mathrm{Id}|_{0} \MP \varepsilon^2 \leq \kappa\varepsilon/4 \]
since $\varepsilon \MP 1$. Therefore $\Phi_2$ is well-defined on $\mathcal{D}^*(\kappa \varepsilon/4)$ and of class $C^{l'+1}$, and so $\Phi : \mathcal{D}^*(\kappa \varepsilon/4) \rightarrow \mathcal{D}^*(\kappa \varepsilon)$ is of class $C^{l'+1}$ and satisfies 
\[|\Phi-\mathrm{Id}|_{0} \leq |\Phi_1-\mathrm{Id}|_{0} + |\Phi_2-\mathrm{Id}|_{0} \leq  3\kappa\varepsilon/4.\] 
If we set
\[ f''=\varepsilon^{-1} f'^{K}-\varepsilon^{-1} \bar{f}'^{K}+R, \]
then it can be proved that $|f''|_{l'} \MP 1$ and we have
\[ H_1 \circ \Phi_2= h+\varepsilon \bar{f}+\varepsilon^2 \bar{f}'+\varepsilon^3 f''\]
hence
\[ H \circ \Phi=h+\varepsilon \bar{f}+\varepsilon^2 \bar{f}'+\varepsilon^3 f''. \]
This concludes the proof.
\end{proof}

\subsection{Proof of Theorem~\ref{thm2} and Theorem~\ref{thm2-potential}}\label{s23}

\subsubsection{Proof of Theorem~\ref{thm2}}

The proof of Theorem~\ref{thm2} is now a consequence of the normal form Proposition~\ref{normal}. Since the latter is defined on a domain which contains the segment $S^*$, whose length is independent of $\varepsilon$, it will be possible to prove the statement of Theorem~\ref{thm2} for the normal form $H \circ \Phi$ by analyzing directly the equation of motions, and using the fact that $\Phi$ is $\varepsilon$-close to the identity, we will prove that the statement remains true for $H$.

\begin{proof}[Proof of Theorem~\ref{thm2}]
Recall that we are considering $H=h+\varepsilon f$ defined on $\T^2 \times B_R$, with $h \in C_1^4(B_R)$ satisfying $(B.1)$ and $(B.2)$ and $f \in \mathcal{F}^7$, so we can apply Proposition~\ref{normal} with $l=1$: there exist positive constants $\varepsilon_0'$ and $C'$ depending only on $R$, the length of $S_1^*$ and $\varpi$ such that if $\varepsilon \leq \varepsilon_0'$, there exists 
a symplectic embedding 
\[
\Phi : \mathcal{D}^*(\kappa \varepsilon/2) \rightarrow 
\mathcal{D}^*(\kappa \varepsilon) 
\]
of class $C^2$ such that
\[ 
H \circ \Phi = 
h+ \varepsilon \bar{f}+ \varepsilon^2 f', \quad 
\bar{f}(\theta_1,I)=\int_{\T}f(\theta_1,\theta_2,I)d\theta_2 
\]
and we have the following estimates
\begin{equation}\label{estfinal}
|\Phi-\mathrm{Id}|_{0} \leq \kappa\varepsilon/2, \quad |f'|_{1} \leq C' 
\end{equation}
where $S^*(\kappa\varepsilon) \subseteq B_R$ and $\mathcal{D}^*(\kappa\varepsilon) \subseteq \T^2 \times B_R$ have been defined in \S\ref{s22}.

Let us consider the Hamiltonian $\tilde{H}=H \circ \Phi$ defined on $\mathcal{D}^*(\kappa\varepsilon/2)$, and we shall write $\Phi(\tilde{\theta},\tilde{I})=(\theta,I)$. Since $f \in \mathcal{F}^7$, there exist $I^*=(I_1^*,0)$ in the interior of $S^*$ and $\theta_1^* \in \T$ such that
\begin{equation}\label{drift}
|\partial_{\tilde{\theta_1}} \bar{f}^*(\theta_1^*)|=|\partial_{\tilde{\theta_1}}\bar{f}(\theta_1^*,I^*)|\geq \lambda.
\end{equation}     
Note that necessarily $\lambda\leq 1$ since $\bar{f}\in C^7_1(\T^2 \times B_R)$, and recall that $\delta^{*}$ is the distance of $I^*$ to the boundary of $S^*$. Since $\tilde{H}$ is $C^2$, we can consider a solution $(\tilde{\theta}(t),\tilde{I}(t))$ of the system defined by $\tilde{H}$ with an initial condition $(\tilde{\theta}(0),\tilde{I}(0))$ such that $\tilde{I}(0)=I^*$, $\tilde{\theta_1}(0)=\theta_1^*$ and $\tilde{\theta_2}(0) \in \T$ is arbitrary: we have the equations
\begin{equation}\label{mouve}
\begin{cases}
\frac{d}{dt}\tilde{I}_1(t)=-\partial_{\tilde{\theta_1}} \tilde{H}(\tilde{\theta}(t),\tilde{I}(t))=-\varepsilon\partial_{\tilde{\theta_1}} \bar{f}(\tilde{\theta}_1(t),\tilde{I}(t)) -\varepsilon^2\partial_{\tilde{\theta_1}} f'(\tilde{\theta}(t),\tilde{I}(t)), \\
\frac{d}{dt}\tilde{I}_2(t)=-\partial_{\tilde{\theta_2}} \tilde{H}(\tilde{\theta}(t),\tilde{I}(t))=-\varepsilon^2\partial_{\tilde{\theta_2}} f'(\tilde{\theta}(t),\tilde{I}(t)), \\
\frac{d}{dt}\tilde{\theta}_1(t)=\partial_{\tilde{I_1}} \tilde{H}(\tilde{\theta}(t),\tilde{I}(t))=\omega_1(I(t))+\varepsilon\partial_{\tilde{I_1}} \bar{f}(\tilde{\theta}_1(t),\tilde{I}(t))+\varepsilon^2\partial_{\tilde{I_1}} f'(\tilde{\theta}(t),\tilde{I}(t)), \\
\frac{d}{dt}\tilde{\theta}_2(t)=\partial_{\tilde{I_2}} \tilde{H}(\tilde{\theta}(t),\tilde{I}(t))=\omega_2(I(t))+\varepsilon\partial_{\tilde{I_2}} \bar{f}(\tilde{\theta}_1(t),\tilde{I}(t))+\varepsilon^2\partial_{\tilde{I_2}} f'(\tilde{\theta}(t),\tilde{I}(t)),
\end{cases}
\end{equation}
since $\bar{f}$ is independent of the second angular variables. For a positive constant $\delta$ to be chosen later in terms of $\lambda$ and $\delta^*$, we let $\tau=\delta\varepsilon^{-1}$. From the second equation of~\eqref{mouve} and the first estimate of~\eqref{estfinal}, we get
\[ |\tilde{I}_2(t)-\tilde{I}_2(0)|=|\tilde{I}_2(t)|\leq C'\varepsilon\delta, \quad |t|\leq \tau, \]
which makes sense provided that $|\tilde{I}_2(t)-\tilde{I}_2(0)| \leq \kappa\varepsilon/2$ for $|t|\leq\tau$, and this is satisfied if $C'\delta\leq \kappa/2$, that is $\delta \leq \kappa(2C')^{-1}$. Now for $|t|\leq \tau$, recalling that $\omega_1(\tilde{I}_1(t),\tilde{I}_2(0))=0$ and $h \in C^4_1(B_R)$, we have
\[ |\omega_1(\tilde{I}(t))|=|\omega_1(\tilde{I}_1(t),\tilde{I}_2(t))|=|\omega_1(\tilde{I}_1(t),\tilde{I}_2(t))-\omega_1(\tilde{I}_1(t),\tilde{I}_2(0))|\leq |\tilde{I}_2(t)-\tilde{I}_2(0)| \leq C'\varepsilon\delta.  \]
Therefore, from the third equation of~\eqref{mouve}, the second estimate of~\eqref{estfinal} and the fact that $\bar{f} \in C^7_1(\T \times B_R)$, we have
\[ \left|\frac{d}{dt}\tilde{\theta}_1(t)\right|\leq C'\varepsilon\delta+\varepsilon+C'\varepsilon^2=(C'\delta+1+C'\varepsilon)\varepsilon\leq C\varepsilon, \quad |t|\leq \tau \]
with $C=1+2C'$, provided $\delta\leq 1$ and since $\varepsilon \leq 1$. This implies that
\begin{equation}\label{vart}
|\tilde{\theta}_1(t)-\theta_1^*|=|\tilde{\theta}_1(t)-\tilde{\theta}_1(0)|\leq C\delta, \quad |t|\leq \tau. 
\end{equation}
Moreover, recall that $|\tilde{I}_2(t)-\tilde{I}_2(0)| \leq C'\varepsilon\delta \leq C\delta$ for $|t|\leq \tau$ by the definition of $C$ and since $\varepsilon \leq 1$, and from the first equation of~\eqref{mouve}, the first estimate of~\eqref{estfinal} and the fact that $\bar{f} \in C^7_1(\T \times B_R)$, we also have 
\[ |\tilde{I}_1(t)-I^*_1|=|\tilde{I}_1(t)-\tilde{I}_1(0)|\leq \delta +C'\varepsilon\delta, \quad |t|\leq \tau, \]
which makes sense if $\delta \leq \delta^*$ as this implies that $|\tilde{I}_1(t)-I^*_1| \leq \delta^*+\kappa\varepsilon/2$. In particular 
\[ |\tilde{I}_1(t)-I^*_1|=|\tilde{I}_1(t)-\tilde{I}_1(0)|\leq C\delta, \quad |t|\leq \tau, \]
by the definition of $C$ and since $\varepsilon\leq 1$ and therefore
\begin{equation}\label{vari}
|\tilde{I}(t)-I^*|=|\tilde{I}(t)-\tilde{I}(0)| \leq C\delta, \quad |t|\leq \tau. 
\end{equation} 
Using the fact that $\bar{f} \in C^7_1(\T \times B_R)$, from~\eqref{vart} and~\eqref{vari} we obtain
\begin{equation}\label{eqvar}
|\partial_{\tilde{\theta}_1}\bar{f}(\tilde{\theta}_1(t),\tilde{I}(t))-\partial_{\tilde{\theta}_1}\bar{f}(\theta_1^*,I^*)| \leq C\delta, \quad |t|\leq \tau. 
\end{equation} 
We eventually choose $\delta=\min\{\lambda(4C)^{-1},\delta^*\}$, and hence $\tau=\delta\varepsilon^{-1}\leq\lambda(4C)^{-1}\varepsilon^{-1}$. We have to make sure that $\delta\leq 1$ and $\delta \leq \min\{\delta^*,\kappa(2C')^{-1}\}$. The first requirement is obviously satisfied since $C\geq 1$ and $\lambda \leq 1$, and hence $\delta \leq \lambda(4C)^{-1} \leq 1$. For the second one, which reduces to $\delta \leq \kappa(2C')^{-1}$, note that $\varpi \leq 1$ since $h \in C^4_1(B_R)$, so $\kappa \geq 1$ hence $\lambda \leq 1 \leq 2\kappa $ and this implies that $\delta \leq \kappa(2C')^{-1}$ as $C \geq C'$. Now from~\eqref{drift},~\eqref{eqvar} and the definition of $\delta$, we have for all $|t|\leq \tau$,
\[ |\varepsilon\partial_{\tilde{\theta}_1}\bar{f}(\tilde{\theta}_1(t),\tilde{I}(t))|\geq |\varepsilon\partial_{\tilde{\theta}_1}\bar{f}(\theta_1^*,I^*)|-|\varepsilon\partial_{\tilde{\theta}_1}\bar{f}(\tilde{\theta}_1(t),\tilde{I}(t))-\varepsilon\partial_{\tilde{\theta}_1}\bar{f}(\theta_1^*,I^*)|\geq \varepsilon\lambda-C\varepsilon\delta \geq 3\varepsilon\lambda/4. \]
Moreover, if we assume that $\varepsilon \leq (4C')^{-1}\lambda$, then from the second estimate of~\eqref{estfinal}, we have
\[ |\varepsilon^2\partial_{\tilde{\theta_1}} f'(\tilde{\theta}(t),\tilde{I}(t))|\leq C'\varepsilon^2 \leq \varepsilon\lambda/4, \quad |t|\leq \tau, \]
and this gives, as before,
\[ |\varepsilon\partial_{\tilde{\theta}_1}\bar{f}(\tilde{\theta}_1(t),\tilde{I}(t))+\varepsilon^2\partial_{\tilde{\theta_1}} f'(\tilde{\theta}(t),\tilde{I}(t))|\geq 3\varepsilon\lambda/4 - \varepsilon\lambda/4=\varepsilon\lambda/2, \quad |t|\leq \tau. \]
Now from the first equation of~\eqref{mouve}, we obtain
\[ \left|\frac{d}{dt}\tilde{I}_1(t)\right|\geq \varepsilon\lambda/2, \quad |t|\leq \tau, \]
which eventually gives
\[ |\tilde{I}_1(\tau)-\tilde{I}_1(0)|\geq \tau \varepsilon\lambda/2\geq 2C \delta^2. \]

Coming back to the original Hamiltonian, $\Phi(\tilde{\theta}(t),\tilde{I}(t))=(\theta(t),I(t))$ is a solution of the Hamiltonian $H$, and from the first estimate of~\eqref{estfinal}, we have
\[ |\tilde{I}_1(t)-I_1(t)| \leq \kappa\varepsilon/2, \quad |\tilde{I}_2(t)-I_2(t)| \leq \kappa\varepsilon/2 \]
as long as $\tilde{I}(t)\in S^*(\kappa\varepsilon/2)$, so in particular
\[ |\tilde{I}_1(0)-I_1(0)| \leq \kappa\varepsilon/2, \quad |\tilde{I}_1(\tau)-I_1(\tau)| \leq \kappa\varepsilon/2. \]
Assuming that $\varepsilon \leq C\delta^2/\kappa$, this gives
\[ |I_1(\tau)-I_1(0)|\geq |\tilde{I}_1(\tau)-\tilde{I}_1(0)|-|\tilde{I}_1(\tau)-I_1(\tau)|-|\tilde{I}_1(0)-I_1(0)| \geq 2C \delta^2 -\kappa\varepsilon \geq C\delta^2.   \]
Summing up, if we define
\[ \varepsilon_0=\min\{\varepsilon_0',\lambda(4C')^{-1},C\delta^2\kappa^{-1}\} \]
and $c=\kappa$, then for $\varepsilon\leq\varepsilon_0$, $\delta=\min\{\delta^*,\lambda(4C)^{-1}\}$ and $\tau=\delta\varepsilon^{-1}$, the Hamiltonian $H$ has a solution $(\theta(t),I(t))$ for which 
\[ |I_1(0)-I_1^*| \leq c\varepsilon, \quad |I_1(\tau)-I_1(0)|\geq \lambda^2(16C)^{-1}=C\delta^2.  \]
Moreover, for all $t \in [0,\tau]$, 
\[ |I_2(t)|\leq  c\varepsilon, \quad d(I_1(t),S_1^*) \leq c\varepsilon.\] 
This was the statement to prove.
\end{proof}

\subsubsection{Proof of Theorem~\ref{thm2-potential}}

The proof of Theorem~\ref{thm2-potential} is based on the normal form Proposition~\ref{normal2}, and is similar to the proof of Theorem~\ref{thm} so we will not repeat several details. The only difference is that the perturbation $f$ (and hence its average $\bar{f}$) is action independent so that the action dependence in the normal form is of order $\varepsilon^2$, and moreover the dependence on $\theta_2$ in the normal form is of order $\varepsilon^3$. These facts will be used to control the solution on a longer time $\tau$, and this will eventually give a larger drift of the action variable $I_1$.

\begin{proof}[Proof of Theorem~\ref{thm2-potential}]
Recall that we are considering $H=h+\varepsilon f$ defined on $\T^2 \times B_R$, with $h \in C_1^{10}(B_R)$ satisfying $(B.1)$ and $(B.2)$ and $f \in \mathcal{G}^{19}$, so we can apply Proposition~\ref{normal} with $l=7$ and $l'=1$: there exist positive constants $\varepsilon_0''$, $C'$ and $C''$ depending only on $R$, the length of $S_1^*$ and $\varpi$ such that if $\varepsilon \leq \varepsilon_0'$, there exists 
a symplectic embedding 
\[
\Phi : \mathcal{D}^*(\kappa \varepsilon/4) \rightarrow 
\mathcal{D}^*(\kappa \varepsilon) 
\]
of class $C^{2}$ such that
\[ 
H \circ \Phi = 
h+ \varepsilon \bar{f}+ \varepsilon^2 \bar{f}'+\varepsilon^3 f'', \quad 
\bar{f}(\theta_1,I)=\int_{\T}f(\theta_1,\theta_2,I)d\theta_2, \quad \bar{f}'(\theta_1,I)=\int_{\T}f'(\theta_1,\theta_2,I)d\theta_2 
\]
and we have the following estimates
\[ |\Phi-\mathrm{Id}|_{0} \leq 3\kappa\varepsilon/4, \quad |f'|_{7} \leq C', \quad |f''|_{1} \leq C''. \] 
Let $I' \in S^*$ and $I'' \in S^*$, without loss of generality we may assume that $I'_1 \leq I_1''$ and recall that $\rho=|I'-I''|=I_1''-I_1'$. As before, we consider a solution $(\tilde{\theta}(t),\tilde{I}(t))$ of the system defined by $\tilde{H}=H \circ \Phi$ with an initial condition $(\tilde{\theta}(0),\tilde{I}(0))$ such that $\tilde{I}(0)=I'$, $\tilde{\theta_1}(0)=\theta_1^*$ and $\tilde{\theta_2}(0) \in \T$ is arbitrary, and for a positive constant $\delta$ to be chosen later in terms of $\lambda$ and $\rho$, we let $\tau'=\delta\varepsilon^{-1}$. Moreover, we define
\[ \tilde{\tau}=\inf \{t \geq 0, \; | \; |\tilde{I}_1(t)-I_1'|\geq\rho\}\in \R^+ \cup \{+\infty\} \]
and we let $\tau=\min\{\tau',\tilde{\tau}\} \in \R^+$. Reversing time if necessary, we may assume that $I'_1 \leq \tilde{I}_1(\tau)$. Note that in the normal form, only $f''$ depends on $\theta_2$, so from the equations of motion of $\tilde{H}$ and the estimate on $f''$, 
\[ |\tilde{I}_2(t)-\tilde{I}_2(0)|=|\tilde{I}_2(t)|\leq C'\varepsilon^2\delta, \quad |t|\leq \tau, \]
which makes sense provided that $|\tilde{I}_2(t)-\tilde{I}_2(0)| \leq \kappa\varepsilon/4$ for $|t|\leq\tau$, and this is satisfied if $C'\delta\varepsilon\leq \kappa/4$, that is $\varepsilon \leq \kappa(4\delta C')^{-1}$. This implies, as before, that 
\[ |\omega_1(\tilde{I}(t))| \leq |\tilde{I}_2(t)-\tilde{I}_2(0)| \leq C'\varepsilon^2\delta  \]
and hence, using the fact that $\bar{f}$ is action-independent,
\[ \left|\frac{d}{dt}\tilde{\theta}_1(t)\right|\leq C''\varepsilon^2\delta+C'\varepsilon^2+C''\varepsilon^3=(C''\delta+C'+C''\varepsilon^{2})
\varepsilon^{2}\leq C\varepsilon^{2}, \quad |t|\leq \tau \]
with $C=C''\delta+C'+C''$, since $\varepsilon \leq 1$. Hence 
\begin{equation}
|\tilde{\theta}_1(t)-\theta_1^*|=|\tilde{\theta}_1(t)-\tilde{\theta}_1(0)|\leq C\varepsilon\delta, \quad |t|\leq \tau. 
\end{equation}
Since $\bar{f} \in C^{19}_1(\T \times B_R)$, we obtain from the last estimate that
\begin{equation}\label{eqvar2}
|\partial_{\tilde{\theta}_1}\bar{f}(\tilde{\theta}_1(t))-\partial_{\tilde{\theta}_1}\bar{f}(\theta_1^*)| \leq C\varepsilon\delta, \quad |t|\leq \tau. 
\end{equation} 
Now recall that since $f \in \mathcal{G}^{19}$, we have
\[ |\varepsilon\partial_{\tilde{\theta}_1}\bar{f}(\theta_1^*)| \geq \varepsilon\lambda \] 
and assuming that $\varepsilon \leq \lambda (4C\delta)^{-1}$, we have $C\varepsilon\delta \leq\lambda/4$ so that the last estimate together with~\eqref{eqvar2} implies that
\[ |\varepsilon\partial_{\tilde{\theta}_1}\bar{f}(\tilde{\theta}_1(t))| \geq 3\varepsilon\lambda/4, \quad 0 \leq t \leq \tau. \] 
Moreover, if we assume that $\varepsilon \leq \lambda(8C')^{-1}$ and $\varepsilon \leq \lambda^{1/2}(8C'')^{-1/2}$, then
\[ |\varepsilon^2\partial_{\tilde{\theta_1}} \bar{f}'(\tilde{\theta}(t),\tilde{I}(t))|\leq C'\varepsilon^2 \leq \varepsilon\lambda/8, \quad |t|\leq \tau, \]
and
\[ |\varepsilon^3\partial_{\tilde{\theta_1}} f''(\tilde{\theta}(t),\tilde{I}(t))|\leq C''\varepsilon^3 \leq \varepsilon\lambda/8, \quad |t|\leq \tau\]
which gives
\[ |\varepsilon\partial_{\tilde{\theta}_1}\bar{f}(\tilde{\theta}_1(t))+\varepsilon^2\partial_{\tilde{\theta_1}} \bar{f}'(\tilde{\theta}(t),\tilde{I}(t))+\varepsilon^3\partial_{\tilde{\theta_1}} f''(\tilde{\theta}(t),\tilde{I}(t))|\geq 3\varepsilon\lambda/4 - \varepsilon\lambda/8-\varepsilon\lambda/8=\varepsilon\lambda/2, \quad |t|\leq \tau. \]
From the equations of motion of $\tilde{H}$, this implies
\[ \left|\frac{d}{dt}\tilde{I}_1(t)\right|\geq \varepsilon\lambda/2, \quad |t|\leq \tau, \]
and hence
\[ |\tilde{I}_1(\tau)-\tilde{I}_1(0)|=\tilde{I}_1(\tau)-\tilde{I}_1(0)\geq \tau \varepsilon\lambda/2. \]
We have imposed no restriction on the choice of $\delta$, so we eventually choose $\delta=2\rho\lambda^{-1}$. If $\tau=\tilde{\tau}$, then by definition of $\tilde{\tau}$ we have $\tilde{I}_1(\tau)-I'=\rho$, so $\tilde{I}(\tau)=I''$ and we have $\tau \leq\tau'=\delta\varepsilon^{-1}$. If $\tau=\tau'=\delta\varepsilon^{-1}$, then from the last estimate and the choice of $\delta$ we obtain
\[ \tilde{I}_1(\tau)-\tilde{I}_1(0)\geq \delta\lambda/2=\rho \]
which implies that $\tilde{I}_1(\tau)-\tilde{I}_1(0)=\rho$ and $\tau = \tilde{\tau}$, and hence $\tilde{I}(\tau)=I''$. Therefore, in any cases, $\tilde{\tau}$ is finite, and we have $\tilde{I}(\tau)=I''$ for $\tau \leq \delta\varepsilon^{-1}$.  
Coming back to the original Hamiltonian, $\Phi(\tilde{\theta}(t),\tilde{I}(t))=(\theta(t),I(t))$ is a solution of the Hamiltonian $H$, and using the estimate on $\Phi$, we have
\[ |\tilde{I}_1(t)-I_1(t)| \leq 3\kappa\varepsilon/4, \quad |\tilde{I}_2(t)-I_2(t)| \leq 3\kappa\varepsilon/4 \]
as long as $\tilde{I}(t)\in S^*(\kappa\varepsilon/2)$, so in particular
\[ |\tilde{I}_1(0)-I_1(0)| \leq 3\kappa\varepsilon/4, \quad |\tilde{I}_1(\tau)-I_1(\tau)| \leq 3\kappa\varepsilon/4. \]
Summing up, if we define
\[ \varepsilon_0=\min\{\varepsilon_0'',\kappa(4C'\delta)^{-1},\lambda (4C\delta)^{-1},\lambda(8C')^{-1},\lambda^{1/2}(8C'')^{-1/2}\} \]
and $c=\kappa$, then for $\varepsilon\leq\varepsilon_0$, $\delta=2\rho\lambda^{-1}$ and $\tau\leq\delta\varepsilon^{-1}$, the Hamiltonian $H$ has a solution $(\theta(t),I(t))$ for which 
\[ |I_1(0)-I_1'| \leq c\varepsilon, \quad |I_1(\tau)-I_1''| \leq c\varepsilon.  \]
Moreover, for all $t \in [0,\tau]$, 
\[ |I_2(t)|\leq c\varepsilon, \quad d(I_1(t),S_1^*)\leq c\varepsilon.\] 
This was the statement to prove.
\end{proof}

\appendix

\section{Technical estimates}\label{app}

Let $S^*$ be a bounded domain in $\R^2$, and for $0<\varepsilon<1$ and a positive constant $\kappa$, consider the domains $S^*(\kappa\varepsilon)=\{I \in \R^2 \; | \; d(I,S^*)\leq \kappa\varepsilon\}$ and $\mathcal{D}(\kappa\varepsilon)=\T^2 \times S^*(\kappa\varepsilon)$.

Let us begin by recalling some elementary estimates. First if $f\in C^r(\mathcal{D}(\kappa\varepsilon))$ for $r\geq 2$, then for $j\in \N^4$, $|j|\leq r$, $\partial^l f \in C^{r-|j|}(\mathcal{D}(\kappa\varepsilon))$ and obviously
\begin{equation}\label{eg1}
|\partial^l f|_{r-|j|}\leq |f|_{r}.
\end{equation}
In particular, this implies that if $f\in C^r(\mathcal{D}(\kappa\varepsilon))$, then its Hamiltonian vector field $X_f$ is of class $C^{r-1}$ and  
\[ |X_f|_{r-1}\leq |f|_{r}. \]
Then, given two functions $f,g\in C^r(\mathcal{D}(\kappa\varepsilon))$, the product $fg$ belongs to $C^r(\mathcal{D}(\kappa\varepsilon))$ and by the Leibniz formula
\begin{equation}\label{eg2}
|fg|_{r} \leq c(r) |f|_{r}|g|_{r}.
\end{equation}
for some constant depending only on $r$. By~\eqref{eg1} and~\eqref{eg2}, the Poisson Bracket $\{f,g\}$ belongs to $C^{r-1}(\mathcal{D}(\kappa\varepsilon))$ and
\begin{equation}\label{poisson}
|\{f,g\}|_{r-1} \leq c(r) |f|_{r}|g|_{r}.
\end{equation}
for another constant $c(r)$ depending only on $r$.

Given any two vector-valued functions $F$ and $G$ of class $C^r$, defined on appropriate domains in $\R^m$ with values in $\R^m$, such that the composition $F \circ G$ makes sense, from Faa di Bruno's formula (see for instance \cite{AR67}) it is easy to see that $F \circ G$ is of class $C^r$ and
\begin{equation}\label{Faa}
|F \circ G|_r \leq c(m,r) |F|_r |G|_r^r.
\end{equation}
for a constant $c(m,r)$ depending only on $m$ and $r$. Faa di Bruno's formula and classical results on the existence and regularity of solutions of differential equations can be used to prove the following lemma.

\begin{lemma}\label{tech} 
Let $\chi\in C^{r+2}(\mathcal{D}(\kappa\varepsilon))$, and assume that
\begin{equation*}
|\chi|_{r+1}\leq \kappa\varepsilon/2, \quad \varepsilon \leq c  
\end{equation*}
for some positive constant $c$. Then, for all $|t|\leq 1$, $X^t_\chi : \mathcal{D}(\kappa\varepsilon/2) \rightarrow \mathcal{D}(\kappa\varepsilon)$ is a well-defined symplectic embedding of class $C^{r+1}$, and we have the estimates
\[ |X^t_\chi-\mathrm{Id}|_0 \leq \kappa\varepsilon/2, \quad |X^t_\chi|_r \leq C  \]
for some constant $C$ depending only on $c$, $\kappa$ and the diameter of $S^*$.
\end{lemma} 

Note that the constants $C$ depend only on $c$, $\kappa$ and on the diameter of $S^*(\kappa\varepsilon)$, but since $\varepsilon<1$, the latter is bounded by $d+2\kappa$ where $d$ is the diameter of $S^*$.

The proof of the above lemma is a simple adaptation of Lemma $3.15$ in \cite{DH09}, see also Lemma $A.1$ in \cite{BouII12}.

\addcontentsline{toc}{section}{References}
\bibliographystyle{amsalpha}
\bibliography{Nonconvex}

\end{document}